\documentclass[11pt]{article}
\usepackage{mylatexsetting}

\title{Sum the Probabilities to $m$ and Stop}
\author{Zakaria Derbazi  \\{\it\small Queen Mary University of London}}

\begin{document}
	\maketitle
	\begin{abstract}
		This work investigates the optimal selection of the $m$th  last success in a sequence of $n$ independent Bernoulli trials. We propose a threshold strategy that is $\varepsilon$-optimal under minimal assumptions about the monotonicity of the trials' success probabilities. This new strategy ensures stopping at most one step earlier than the optimal rule. Specifically, the new threshold coincides with the point where the sum of success probabilities in the remaining trials equals $m$. 
		We show that the underperformance of the new rule, in comparison to the optimal one,  is of the order $O(n^{-2})$ in the case of a Karamata-Stirling success profile with parameter $\theta > 0$  where $p_k = \theta / (\theta + k - 1)$ for the $k$th trial. We further leverage the classical weak convergence of the number of successes in the trials to a Poisson random variable to derive the asymptotic solution  of the stopping problem. Finally, we present illustrative examples highlighting the close performance between the two rules.
	\end{abstract}
	
	\section{Introduction}
	The last success problem generalises the best choice problem to success probabilities beyond the classical random record model, where the success probability on the $k$th trial is $p_k=k^{-1}$. We refer to success probabilities $\profile\coloneqq (p_1, p_2, \ldots)$ as the \textit{success profile} of the trials (or the problem).
	
	\citeay{Pfeifer} introduced one of the earliest last success problems, where the success profile is based on Nevzorov's $F^\alpha$ record model. In this model, the success probability at trial $k$ is given by $p_k =\alpha_k / (\alpha_1 + \cdots + \alpha_k)$, where the $\alpha$'s are positive numbers known as the \textit{exponents} of the $F^\alpha$ scheme.  Later, \citeay{HillKrengel} formalised the last success problems for arbitrary success profiles, and \citeay{Odds} established the remarkable odds theorem. This result provides a simple algorithm to compute optimal policies that achieve an asymptotic winning probability of $e^{-1}$. The generalisation of the odds theorem to selection problems involving the $m$th  last success began with \citeay{BrussPaindaveine} and \citeay{TamakiOdds}.  Although Pfeifer briefly derived the asymptotic value for the `selecting any of the last $m$ successes' problem, detailed exploration was not provided. The former two studies demonstrated that the optimal policy is of the threshold type. In the case of the $m$th last  success, the threshold is the index where the \textit{multiplicative sum of the odds} of order $m - 1$ exceeds the multiplicative sum of the odds of order $m$. This multiplicative sum of the odds of order $m$ is a `higher order' sum of the odds corresponding to an elementary symmetric polynomial in $m$ different odds. \citeay{MatsuiAnoBound} exploited the properties of these polynomials to derive a lower bound for the winning probability. Subsequently, in 2017, they presented a unified approach for identifying any success within a contiguous range of last successes in independent Bernoulli trials. This framework generalises the $m$th last-success problem and the `any of the last $m$ successes' problem of Pfeifer and Tamaki. However, establishing the unimodality of the winning probability required tedious algebraic manipulation involving elementary symmetric polynomials.
	
	\citeay{D3} demonstrated that the unimodality of the winning probability for the general problem of Matsui and Ano can be derived using a probabilistic argument anchored by two principal insights. The first insight relates to the characteristics of the Poisson-binomial distribution, which represents the number of successes in inhomogeneous Bernoulli trials. The second insight involves the \textit{total positivity} of the Markov chain embedded in the success states of the trials (see \citeay{KarlinTPBook}). This property ensures the unimodality of the probability of choosing a success among the $\ell$-th to $m$th last successes, where $1 \le \ell \le m < n$, in $n$ independent Bernoulli trials.

	Based on these insights, we propose a new threshold-based stopping strategy that achieves asymptotic optimality by refining the results of \citeay{Samuels1965}. Samuels, in turn, improved  \citeay{Darroch}'s rule on the mode. This rule states that the mode of the Poisson-binomial distribution deviates from the mean by at most one. The threshold of the new rule stops at the first success following a critical index that satisfies two conditions:
	
	\begin{enumerate}
		\item The success probabilities of the trial at this threshold and the immediately subsequent trial are the largest and the second largest,  respectively, among the remaining trials and
		\item The sum of success probabilities in the remaining trials is at least $m$.
	\end{enumerate}
	In actuality, the first condition is stronger than what is required by our result. Nonetheless, it illustrates that the monotonicity of the success probabilities was not relied upon to derive the new strategy.
	
	This new strategy's inherent simplicity translates to significant advantages in terms of  ease of implementation and computational efficiency compared to the original solution based on elementary symmetric polynomials. Furthermore, we rely on the weak convergence of the number of successes to a Poisson random variable to derive the asymptotic winning probability and the asymptotic threshold. This approach offers a clear advantage over the analytical method, which necessitates approximating an iterated Riemann sum using a multiple integral.
	
	
	The paper is organised as follows: Section 2 begins with an overview of the problem of selecting the $m$th last success in a sequence of $n$ independent Bernoulli trials. An analytical solution to this problem is then presented. In Section 3, the work of \citeay{Samuels1965} is revisited and a refined interpretation of his findings is provided. This refined interpretation allows us to identify the mode of the probability function and introduce the mean rule, a stopping rule that is asymptotically optimal and based on the sum of probabilities. Moving on to Section 4, we employ the classical Poisson approximation of the Poisson-binomial distribution to derive an asymptotic threshold rule. This rule is identical to the one obtained through the analytical approach. Finally, Section 5 presents a series of examples to compare the thresholds and winning probabilities of the two rules.
	
	\section{The Analytical method}
	Let $X_1, X_2, \ldots$ be a sequence of independent Bernoulli trials defined on a complete probability space $(\Omega, {\mathcal B}, \proba)$ and let $ {\cal F}_k \coloneqq \sigma(X_1, \ldots, X_k)$ denote the sigma-algebra generated by the first $k$ trials. Define ${\mathcal T}$ to be the class of all nonanticipating stopping rules  such that $\tau \in {\cal T}$ means that $\{ \tau = k\} \in \mf_k$ and $X_\tau = 1$. This class also includes the rules that never stop.
	We examine the  optimal selection of the $m$th last success in $n$ Bernoulli trials. This problem was proposed and solved by Bruss and Paindaveine analytically, and it can be described by the following stopping time
	\begin{alignat}{2}\label{stopping.time}
		\tau_m &= \argmax_{\tau \in {\cal T}} ~\proba\left[\sum_{j=\tau}^n X_j= m\right].
	\end{alignat}
	We introduce some necessary notations to lay the groundwork for treating this problem. Let $\profile\coloneqq(p_1,p_2,\ldots, p_n)$ be a general profile of success probabilities associated with trials $X_1, X_2, \ldots, X_n$ where $p_k=\proba(X_k=1)$. Without loss of generality, assume that $0 < p_k < 1$ for all $k\ge1$. 
	
	The number of successes among trials $k, k+1,\cdots,n$ is a random variable
	with a Poisson-binomial distribution. It has the following probability generating function
	$$z \mapsto \prod_{j=k}^n (1-p_j+zp_j).$$
	
	A straightforward computation gives the probability of zero successes and the probability of $m$ successes from stage $k$ onwards as
	\begin{alignat}{2}
		s_0(k, n)&\coloneqq\prod_{j=k}^n (1-p_j),~~~~k<n \label{s0}\\
		s_m(k, n)&\coloneqq s_0(k,n)R_m(k, n),  ~~~~~k\le n-m+1  \label{s0sm},
	\end{alignat}
	where $R_m(k,n)$ is an $m$th elementary symmetric polynomial in variables $(r_k, \ldots, r_n)$ and $r_{k}\coloneqq {p_k}/{(1-p_k)}$ denotes the odds ratio for trial $k$.   The polynomial   $R_m(k,n)$ is given by 
	\begin{equation}\label{Rjk}
		R_{m}(k, n)\coloneqq\sum_{k \le i_1 < \ldots <i_m \le n} r_{i_1}\cdots r_{i_m}, ~~~R_{0}(k,n)=1,
	\end{equation}
	where $R_{m}(k, n) =0$ and $s_m(k,n)=0$ for all  $k > n-m + 1$.
	Based on definition (\ref{Rjk}), It is not difficult to find that 
	\begin{alignat}{2} 
		R_m(k,n) &=  R_{m}(k+1,n) +r_{k}R_{m-1}(k+1, n)\label{Rj(k+1)}.
	\end{alignat}

	Tamaki \cite{TamakiOdds} referred to $R_m(k,n)$ as  the `sum of the multiplicative odds', which also appears in \cite{BrussPaindaveine} for general $m \ge 1$ and in \cite{Odds} as the `sum of the odds' $R_1(k,n)=r_k + \cdots + r_n$. A connection between the three problems and  elementary symmetric  polynomials based on $R_m(k,n)$ is detailed in \cite{MatsuiAno}.
	
	It is not difficult to recognise that at any given trial $k$, the event of exactly $m$ successes corresponds to  1) observing a success at the current trial and further $m-1$ successes in the subsequent trials or 2) observing a failure at the current trial and exactly $m$ successes in the remaining trials. This relationship is captured by the  recursion
	\begin{equation*}
		s_m(k,n)=p_ks_{m-1}(k+1,n) + (1-p_k)s_m({k+1},n).
	\end{equation*}
	By symmetry, obtaining $m$ successes in $n$ trials entails observing a success on the last trial and $m-1$ successes in the preceding ones or a failure on the  $n$th trial and  observing exactly $m$ successes in the first $n-1$ trials. The essence of this description is conveyed through the subsequent identity
	\begin{equation*}
		s_m(k,n)=p_ns_{m-1}(k,n-1) + (1-p_n)s_m({k},n-1).
	\end{equation*}
	
	Let $\Delta_k, \Delta_m$  be {\it backward difference}  operators applied respectively to the first parameter $k$ and the index $m$. The previous recursion leads us to consider increments  in success probabilities between two consecutive stages, which we define as
	\begin{alignat}{2}
		\Delta_k s_m(k, n) &\coloneqq s_m(k, n) - s_m(k-1, n)\nonumberj
		&= p_{k-1}s_0(k,n)\left[R_{m}(k, n)- R_{m-1}(k, n)\right]	\nonumberj	
		&=p_{k-1} \Delta_m s_m(k,n).  \label{operator.identity}
	\end{alignat}

	We now examine the details of the solution to the stopping problem. Recall from definition  (\ref{stopping.time}) that our objective is to maximise the \textit{probability function} $s_m(\cdot, n)$. Lemma 11 in \cite{D3} establishes that a myopic strategy is optimal for last success problems with a unimodal reward function. We leverage this concept to demonstrate the optimality of the threshold rule, which aligns with the myopic strategy in this case.

	Next, we present the following result of \cite{BrussPaindaveine}, which defines the optimal threshold rule for the $m$th  last-success problem.
	\begin{thm}[\citeay{BrussPaindaveine}]\label{thm:unimodality.lastm}
		For the problem of  selecting  the  $m$th  last success in a sequence of $n$ independent Bernoulli trials, an optimal rule exists, and the  policy is to stop at the first success, if any, on or after index $k_m$, where
		\begin{alignat*}{2}
			k_m \coloneqq \min\{1 \le i \le n-m+1: R_m(i, n) \le  R_{m-1}(i, n)\} \vee 1.
		\end{alignat*}
		and  the winning probability is given by $s_m(k_m, n)$.	
	\end{thm}

	Next, we derive the analytical solution when the sequence of success probabilities $p_1, \ldots, p_n$ follows the Karamata-Stirling profile given by
	\begin{equation}\label{profile}
		p_k = \frac{\theta}{\theta+k-1},~~~~ k=1,2,\ldots
	\end{equation}
	where $\theta$ is a positive real number.

	The distribution of the number of successes among trials $k + 1, \ldots, n$ has a nice probability generating function
	\begin{alignat}{2}\label{PGF}
		z \mapsto  \prod_{i=k+1}^{n} \bigg(1-\frac{\theta}{\theta + i-1}+\frac{\theta}{\theta + i-1}z\bigg)= \frac{(k+\theta z)_{n-k}}{(k+\theta)_{n-k}}, 
	\end{alignat}
	where $(x)_n$ denotes the Pochhammer symbol.  Using  the success profile  (\ref{profile}) in expressions (\ref{s0}) and (\ref{s0sm}), or straight from (\ref{PGF}), we obtain the probabilities of zero successes and $m$ successes in trials $k+1, \ldots, n$, respectively
	\begin{alignat}{2}
		s_0(k+1, n)&= \frac{(k)_{n-k}}{(k + \theta)_{n-k}},\label{s0.karamata}\\
		s_m(k+1, n)&\coloneqq s_0(k+1,n) R_m(k+1, n; \theta).\nonumber
	\end{alignat}
	As before,  $R_m$ is the sum of the multiplicative odds with an additional parameter $\theta$ given by
	\begin{equation*}
		R_{m}(k, n; \theta)\coloneqq\sum_{k \le j_1 < \ldots <j_m \le n}  \frac{\theta^m}{(j_1-1)\cdots(j_m-1)},
	\end{equation*}
	with the convention  $R_0(k, n; \theta) = 1$ and $R_m(k, n; \theta) = 0$ for $k > n-m+1$.
	
	In the next theorem, we derive the optimal proportions and  values when the trials have a Karamata-Stirling profile. These results  extend  section 6  of Bruss and Paindaveine \cite{BrussPaindaveine} to this new success profile, aligning with Tamaki's approach in \cite{TamakiOdds}.
	
	\begin{thm}
		Consider $n$ Bernoulli trials $X_1, \ldots, X_n$, where $p_k=\proba(X_k=1)$ is given by (\ref{profile}). As $n \rightarrow \infty$, the optimal proportion of trials to skip, $\kappa_m$, and the asymptotic optimal winning probability, $V_m$, associated with the problem of selecting  the  $m$th last success are given by
		\begin{alignat*}{2}
			\mathrm{\rm(i)}~~\kappa_m &= \exp\left(-m / \theta\right),\\
			\mathrm{\rm(ii)}~~V_m&= \frac{m^m}{m!}\exp\left(-m\right).
		\end{alignat*}
		
	\end{thm}
	
	\begin{proof}
		Fix $m \in \{1,\ldots, n-1\}$ and let $k \in [n]=\{1,2,\ldots,n\}$  denote the current trial index. Observe that the sum of multiplicative odds 	$R_{m}(k, n; \theta)$ is an iterated Riemann sum of the function $f(x) = 1/x$ with increments $ \Delta x=1/n$. For $i \in [m]$, the integration  bounds for each integral are defined as $x_{i-1}= \lim_{n \to \infty} j_i / n$ and 1, where $j_1, \ldots, j_m$ denote the indices of summation satisfying the conditions: $j_0 = k$ and $j_{i-1} \le j_i\le n$. We denote the first integration bound $x_0(k(n)) \coloneqq x_0 = k/n$ to mark its dependence on $k$ and $n$. Explicitly, the approximation by a multiple integral becomes
		\begin{alignat}{2}\label{optimal.R}
			R_m(k+1, n, 1)&\coloneqq \theta^m\sum_{j_1}^{n} \frac{1}{j_1}\sum_{j_2=j_1}^{n}\frac{1}{j_2}\cdots\sum_{j_m=j_{m-1}}^{n}\frac{1}{j_m} \nonumberj
			&\rightarrow  \theta^m\int_{x_0}^{1} \frac{\dx[x_1]}{x_1}  \int_{x_1}^{1} \frac{\dx[x_2]}{x_2}\cdots \int_{x_{m-1}}^{1} \frac{\dx[x_m]}{x_m} \nonumberj			
			&=  \frac{(-\theta)^m}{m!} \log^m\left(x_0(k(n))\right).
		\end{alignat}
		The last equality follows from repetitively applying  the formula $$\displaystyle\int_{s}^1 \log^n(x)/x \,\dx[x] = -\frac{\log^{n+1}(s)}{(n+1)}.$$
		To derive the optimal proportion of trials to skip, let $k_m(n)$  be the optimal threshold and let$$\kappa_m\coloneqq  x_0(k_m(n))= \lim_{n \rightarrow \infty} \frac{k_m(n)}{n}, $$ 
		be the corresponding proportion to skip.
		To compute $\kappa_m$, we asymptotically solve  the equation $R_m(k_m(n)+1, n, \theta)=R_{m-1}(k_m(n)+1, n, \theta)$, 
		which by (\ref{optimal.R}) reduces to
		$$\frac{(-\theta)^{m-1}}{(m-1)!} \log^{m-1}\left(\kappa_m\right)\left[-\frac{\theta}{m} \log\left(\kappa_m\right)-1\right] = 0.$$
		
		Consequently, $\log(\kappa_m) = -m/\theta$ and  the proportion of trials skipped before being willing to stop at the $m$th last success is 
		\begin{equation}\label{threshold.km}
			\kappa_m  = \exp\left(\frac{-m}{\theta}\right).
		\end{equation}
		For assertion \rm{(ii)}, the asymptotic winning probability for stopping  at the $m$th last success is 
		\begin{alignat*}{2}
			V_m &\coloneqq \lim_{n \rightarrow \infty} s_m(k_m(n)+1, n).
		\end{alignat*}		
		
		Recall Stirling's asymptotic series where a first-order approximation of a ratio of Gamma functions reads 
		\begin{alignat*}{2}
			\frac{\Gamma(z+a)}{\Gamma(z+b)} &= z^{a-b}\left(1 + O(z^{-1})\right) ~~~\text{as } z\rightarrow \infty.
		\end{alignat*}
		
		Apply this approximation to the two Gamma ratios in $s_0(k+1, n)$, which is defined by (\ref{s0.karamata}) as
		\begin{alignat*}{2}
			s_0(k+1, n) &= \frac{(k)_{n-k}}{(k + \theta)_{n-k}}= \frac{\Gamma(n)\Gamma(\theta+k)}{\Gamma(n+\theta)\Gamma(k)}.
		\end{alignat*}
		
		As $k_m(n) = O(n)$ in our case, a straightforward computation gives
		\begin{alignat}{2}\label{s0.asymptotics}
			s_0(k_m(n)+1, n) &= \left(\frac{k_m(n)}{n}\right)^\theta\left(1+ O(n^{-1})\right)
		\end{alignat}
		Set $k=k_m(n)$ in (\ref{s0.asymptotics})  then take the limit as $n \rightarrow \infty$, to obtain $s_0(k_m(n)+1, n) \rightarrow \exp(-m)$. Next, plug (\ref{threshold.km}) into  (\ref{optimal.R}) to  secure the limit $m^m / m!$. Combining these two limits yields
		\begin{alignat*}{2}
			V_m \rightarrow \frac{m^m}{m!}\exp\left(-m\right).
		\end{alignat*}
		This completes the proof.		
	\end{proof}

	\begin{rem}	
		Matsui and Ano (2017) employed elementary symmetric polynomials to arrive at the same conclusion.
	\end{rem}
	\section{The Probabilistic method}
	Throughout this section, all Bernoulli trials have success profile $\profile$ consisting of known success probabilities $p_1, \ldots, p_n$.
	
	The following result, adapted from Section 3 in \cite{BrussPaindaveine},  is the key result upon which Bruss and Paindaveine rely for establishing Theorem \ref{thm:unimodality.lastm}. 
	\begin{assertion}[Bruss and Paindaveine (2000)]\label{lem:unimodality.lastm}
		If there exists an integer $k_m \in \{1, \ldots, n-m+1\}$, such that $\Delta_m s_{m}(k_m, n) > 0 $  holds then
		\begin{itemize}
			\item[\rm(i)] $\Delta_m s_{m-1}(k, n) > 0$~  for all ~$k \le k_m$.			
			\item[\rm(ii)]  $s_{m}(\cdot, n)$ is unimodal.
		\end{itemize}
	\end{assertion}
	Both \cite{Darroch} and  \cite{BrussPaindaveine}  relied on induction to prove part \rm(i) of Proposition \ref{lem:unimodality.lastm}. Darroch relied on the log-concavity of the Poisson binomial distribution, that is, the log-concavity of $s_m(k, n)$ as a function of $m$,  to study the most likely number of successes in inhomogeneous Bernoulli trials. This property of the Poisson-binomial distribution was established by Bonferroni\cite{Bonferroni} \footnote{This was pointed out to me by Sasha Gnedin. An earlier version of this article mistakenly attributed this discovery to Paul L{\'e}vy. } in 1933   While Bruss and Paindaveine  demonstrated the unimodality of the probability function $s_m(\cdot, n)$, a key observation has been overlooked. Roughly speaking, `the unimodality in $m$ and the unimodality in $k$' are interconnected by the crucial identity (\ref{operator.identity}) and can be inferred from one another. 
	
	The lengthy proof of the unimodality  of $s_m(\cdot,n)$ can be circumvented by considering the inherent nature of the Poisson-binomial distribution. As  a P{\'o}lya frequency (PF) distribution, it is  known to be log-concave. L{\'e}vy demonstrated the log-concavity of the Poisson-binomial distribution (p. 88 in \cite{LevyAddition}). He showed that the sequence $(a_m)$, where $a_m=s_m(k, n)$ is the probability of $m$ successes in $n-k+1$ inhomogeneous Bernoulli trials, satisfies the  well-known Newton's inequality  \cite{PitmanPF}
	\begin{equation*}
		a_m^2 \ge \left( 1 + \frac{1}{m}\right)\left( 1 + \frac{1}{n-k+1-m}\right) a_{m-1}a_{m+1}.
	\end{equation*}
	The given inequality is stronger than log-concavity implied by the $\PF{2}$ (PF of order 2) property. As a result, a $\PF{}$ sequence has either a unique mode, denoted $d$, or two consecutive modes such that $a_d = \max_m a_m$ (see p. 284 in \cite{PitmanPF}). Furthermore, the identity (\ref{operator.identity}) states that for a fixed number of trials $n$, there is an equivalence between:
	\begin{itemize}
		\item[\rm(i)] The probability function $s_m(\cdot, n)$ attaining its mode at some index $k_m(n)$, and
		\item[\rm(ii)] The Poisson binomial distribution  of the number of successes in $n-k_m(n)+1$ trials having $m$ as its mode.
	\end{itemize}
	
	This remarkable equivalence provides a crucial foundation for our new stopping rule, which necessitates the development of several auxiliary results. The next theorem, due to Darroch, establishes that the mode of the sequence $(a_m)$ differs from its mean by less than 1.
	
	\begin{thm}[Darroch's rule, Theorem 4 in \cite{Darroch}] \label{ch1.thm.darroch}
		Given $\mu$, the expected number of successes in $n$ Bernoulli trials, the most probable number of successes  $d$, satisfies
		\begin{equation*}
			\begin{dcases}
				d=m & ~~~~ \text { if }~~m \le \mu < m + \frac{1}{m+2},\\
				d=m, \text{ or }  d = m+1 \text{ or both}  & ~~~~ \text { if }~~m + \frac{1}{m+2} \le \mu \le m+1 - \frac{1}{n-m+1},\\										
				d= m+1 & ~~~~ \text { if }~~ m+1 - \frac{1}{n-m+1} < \mu \le m+1,
			\end{dcases}
		\end{equation*}
		where $m \in \{0, 1, \ldots, n\}$, $ \{x\}$ is the fractional part of $x$ and $\integeru{x}$ is its integer part.
	\end{thm}
	
	This result offers a preliminary insight into a potentially simpler approach to identifying the $m$th  last success. In contrast to the (computationally) complex multiplicative odds rule of Bruss and Paindaveine, the new stopping strategy relies on the mean of the Poisson binomial distribution, corresponding to the sum of success probabilities across trials. However, the usefulness of Darroch's rule  has limitations. Even after the mean exceeds $m$, the most likely number of successes (the mode) might have been at $m$ for many trials. To illustrate this situation, consider the Odds Theorem, which suggests  stopping at stage $k_*$ when the sum of the odds exceeds 1 (starting from the end). Since the odds $(r_k)$ are  greater than the probabilities $(p_k)$, the sum of the odds reaches one before the corresponding sum of probabilities does. If the success probabilities at $k_*-1, k_*-2, \ldots$ are  very small, then reaching a total probability of 1 may require many additional trials.
	
	Theorem 1 of Samuels (1965) and its related Corollaries (15)--(17) offer an indirect measure of the discrepancy between the two critical thresholds based on the mean and the mode, respectively. We will adapt these Corollaries to our notation in the next theorem.
	\begin{thm}[Theorem 1, equations (15) -- (17) in \cite{Samuels1965}] \label{samuels.1516}
		Consider a positive integer  $m$. The probabilities of $m$ successes in $n$ Bernoulli trials having profile $\profile $ satisfy the following implications
		\begin{alignat}{2}
			\min_{k \le j \le n} {p_j} + \sum_{j=k}^n p_j  > m &\implies \Delta_m s_m(k,n)  > 0, \label{eqn.15}\\
			\max_{k+1 \le j \le n} {p_j} + \sum_{j=k+1}^n p_j  < m &\implies \Delta_m s_m(k+1,n)  < 0 \label{eqn.16},\\
			m \le \sum_{j=k}^n p_j  \le m+1 &\implies \Delta_m s_m(k,n)  > 0 ~\text{ and }~  \Delta_{m} s_{m+2}(k,n)  < 0 \label{eqn.17}.
		\end{alignat}
	\end{thm}
	
	\noindent Our next result extends Samuels' Theorem by establishing specific conditions under which the  mode of the probability function $s_m(\cdot, n)$ occurs at a particular index. 
	
	\begin{thm}\label{thm.samuels}
		If for a fixed integer $m \in [n-1]$  there exists a trial index $k^* \in  [n-m]$, such that 
		\begin{equation}\label{proof.samuels.eqn}
			\begin{dcases}
				p_{k^*+1} = \displaystyle\max_{k^*+1 \le j \le n} p_j,\\
				p_{k^*+1}  - p_{k^*} < \min_{k^* \le j \le n} p_j, \text{and} \\
				m-\min_{k^* \le j \le n} p_j< p_{k^*} + \cdots + p_n < m  + p_{k^*} - p_{k^*+1},\\
			\end{dcases}
		\end{equation}
		then index $k^*$ coincides with  the location of the mode of the probability function $s_m(\cdot, n)$.
	\end{thm}
	\begin{proof}
		Since $p_n + \cdots + p_{n-m+1} < m$, it is evident that    $k^*\le n-m$ for all $m$. Now,	suppose that there exists an index $k^* \in [n-m]$ satisfying (\ref{proof.samuels.eqn}). 
		Consider the opposite case of the second assumption instead. That is, $p_{k^*+1}  - p_{k^*} \ge \displaystyle\min_{k^* \le j \le n} p_j + \epsilon$, where $\epsilon \ge 0$. Under this condition, the third assumption in (\ref{proof.samuels.eqn}) becomes: $$m-\min_{k^* \le j \le n} p_j< p_{k^*} + \cdots + p_n < m  - \min_{k^* \le j \le n} p_j -\epsilon, $$
		which is a contradiction. Now, use (\ref{eqn.15}) from Theorem \ref{samuels.1516} to establish that, based on the condition $p_{k^*} + \cdots + p_n > m- \displaystyle\min_{k^* \le j \le n} p_j$, it follows that
		\begin{alignat}{2}\label{proof.ineq1}
			\Delta_m s_m(k^*) > 0.
		\end{alignat}
		Next,  the right-hand side of the third inequality in (\ref{proof.samuels.eqn}) implies that $p_{k^*+1} + \cdots + p_n < m- p_{k^*+1}$. However, by the first assumption, $p_{k^*+1}$ is the maximum probability among the remaining trials. Consequently, use (\ref{eqn.16}) this time  to conclude that 
		\begin{alignat}{2}\label{proof.ineq2}
			\Delta_m s_m(k^*+1) < 0.
		\end{alignat}
		By combining (\ref{proof.ineq1}) and (\ref{proof.ineq2}) together with (\ref{operator.identity}), we  immediately obtain 
		\begin{alignat*}{2}
			\begin{dcases}
				\Delta_k s_m(k^*+1) < 0,\\
				\Delta_k s_m(k^*) > 0. 
			\end{dcases}
		\end{alignat*}
		These two inequalities imply that $k^*$ is the location of the mode, which is the desired result.
	\end{proof}
	\begin{cor}\label{cor.samuels}
		If  from some index $k^* \in [n-m]$, the success probabilities $p_k^*$ and $p_{k^*+1}$ are the largest and second-largest in the remaining trials, then the following condition $$m-\min_{k^* \le j \le n} p_j< p_{k^*} + \cdots + p_n < m  + p_{k^*} - p_{k^*+1},$$ is sufficient for $k^*$ to coincide with the index of the mode of the probability function $s_m(\cdot, n)$.
	\end{cor}
	\begin{proof}
		Suppose there exists a $k^*$ such that $p_{k^*+1} = \displaystyle \max_{k^*+1 \le j \le n } p_j$ and $p_{k^*+1} \le p_{k^*}$. This implies that the first two requirements of Theorem \ref{thm.samuels} are fulfilled. The assertion of the Corollary is immediate as we only require the condition $m-\min_{k^* \le j \le n} p_j< p_{k^*} + \cdots + p_n < m  + p_{k^*} - p_{k^*+1}$ to hold.
	\end{proof}	
	
	\begin{cor}
		If  from some index $k^* \in [n-m]$, the success probabilities are decreasing, then the following condition $$m-\min_{k^* \le j \le n} p_j< p_{k^*} + \cdots + p_n < m  + p_{k^*} - p_{k^*+1},$$ is sufficient for $k^*$ to coincide with the index of the mode of the probability function $s_m(\cdot, n)$.
	\end{cor}
	\begin{proof}
		As $p_{k} = \displaystyle \max_{k \le j \le n } p_j$ for each $k \in [n]$, the assertion is proven by applying Corollary \ref{cor.samuels}. 
	\end{proof}	
	
	At the expense of the precise knowledge about the location of the mode, the next Proposition loosens the bounds on the sum of probabilities.
	\begin{thm}\label{lem.samuels}
		Consider the success profile $\profile$. If there exists a critical index $k^* \in [n-m]$ such that 
		\begin{itemize}
			\item[\rm(i)] $p_{k^*+1}  - p_{k^*} < \displaystyle\min_{k^* \le j \le n} p_j$,
			\item[\rm(ii)] $p_{k^*+1}  = \displaystyle\max_{k^*+1 \le j \le n}$, and
			\item[\rm(iii)]$m-p_{k^*} \le p_{k^*+1} + \cdots + p_n < m$,
		\end{itemize}
		then, the location of the mode of the probability function $s_m(\cdot, n)$ is either $k^*$ or $k^*+1$.
	\end{thm}
	\begin{proof}
		We can rewrite assumption \rm(iii) as
		$$m-\displaystyle\min_{k^* \le j \le n} p_j < p_{k^*} + p_{k^*+1} + \cdots + p_n < m  + p_{k^*}. $$
		This implies that $m < p_{k^*} + p_{k^*+1} + \cdots + p_n < m +1 $. By reference to implication (\ref{eqn.17}), this yields $\Delta_m s_m(k,n) > 0$. Making use of (\ref{operator.identity}), we further infer that $\Delta_k s_m(k,n) > 0$ . Consequently, the mode of the probability function $s_m(\cdot, n)$ is at least at index $k^*$.
		To determine the upper bound on the mode's location, we consider three cases based on assumption \rm(iii):
		\begin{itemize}
			\item Case 1: $p_{k^*+1} + \cdots + p_n < m   -p_{k^*+1}$. This inequality only holds when $p_{k^*} > p_{k^*+1}$. In this scenario, the upper bound on the sum of probabilities becomes stricter. Adding $p_k^*$ to both sides of the inequality and considering assumption \rm(iii), we obtain  $$m \le p_{k^*} + p_{k^*+1} + \cdots + p_n < m  + p_{k^*} -p_{k^*+1}.$$ 
			Given that $m < p_{k^*} + p_{k^*+1} + \cdots + p_n < m +1 $, by the previous analysis, the mode of $s_m(\cdot, n)$ is at least at index $k^*$. If we incorporate assumptions \rm(i) and \rm(ii), then Theorem \ref{thm.samuels} ensures that the mode is located precisely at index $k^*$. 
			
			\item Case 2: $p_{k^*+1} + \cdots + p_n = m   -p_{k^*+1}$. This scenario only holds if $p_{k^*} \ge p_{k^*+1}$. In this case, combining the identity $p_{k^*+1} + \cdots + p_n = m   -p_{k^*+1}$ with assumption \rm(ii) yields
			$$\displaystyle\min_{k^*+1 \le j \le n} p_j  + p_{k^*+1} + \cdots + p_n < m.$$
			Consequently, it follows by (\ref{eqn.16}) that $\Delta_m s_m(k^*+1, n) < 0$. Using (\ref{operator.identity}) once more, we obtain $\Delta_k s_m(k^*+1,n) < 0$. Therefore, the mode's location cannot exceed $k^*+1$.
			
			\item Case 3: $p_{k^*+1} + \cdots + p_n > m   -p_{k^*+1}$. In the final setting, the stricter lower bound on the sum of probabilities implies by (\ref{eqn.16}) that $\Delta_m s_m(k^*+1, n) < 0$. A condition we encountered in the second case. Therefore, the mode cannot exceed $k^*+1$.		
		\end{itemize}
		This completes the proof.
	\end{proof}	
	\begin{cor}
		For fixed $m \in [n-1]$, if there exists an index $k^* \in [n-m]$ such that the following conditions are satisfied 
		\begin{itemize}
			\item[\rm(i)] $p_k^*$ and $p_{k^*+1}$ are the largest and second-largest success probabilities in the remaining trials.
			\item[\rm(ii)]$m-p_{k^*} \le p_{k^*+1} + \cdots + p_n < m,$ 
		\end{itemize}
		then the location of the mode of the probability function $s_m(\cdot, n)$ coincides with index $k^*$ or with index $k^*+1$.
	\end{cor}
	\begin{proof}
		The result is immediate by application of Theorem \ref{lem.samuels}.
	\end{proof}
	\begin{cor}
		For fixed $m \in [n-1]$, if the success probabilities $p_1, \ldots, p_n$ are decreasing and there exists an index $k^* \in [n-m]$ such that $$m-p_{k^*} \le p_{k^*+1} + \cdots + p_n < m,$$ then the location of the mode of the probability function $s_m(\cdot, n)$ coincides with index $k^*$ or with index $k^*+1$.
	\end{cor}
	\begin{proof}
		The result is obvious given that $p_k^*$ and $p_{k^*+1}$ are the largest and second-largest success probabilities in the remaining trials.
	\end{proof}
	Having established the necessary refinements to Samuels's findings, we now turn to the main result of this paper.
	
	\begin{thm}[Sum the Probabilities to $m$ and Stop]\label{thm.darroch.problem1}
		For the problem of  selecting  the {$m$th  last success} in a sequence of $n$ independent Bernoulli trials with success profile $\profile$, consider the threshold  policy {$\tau^*$} which calls for stopping on the first success, if any, on or after index $k^\prime(n)$ given by
		\begin{equation}\label{def.threshold.darroch}
			k^\prime(n) \coloneqq \max\{i \in [n-m]:  \sum_{k = i}^n p_k \ge m \} \vee 1
		\end{equation}
		with $\max \{ \varnothing\} =-\infty.$  If {the probabilities} $p_{k^\prime(n)} $ and $p_{k^\prime(n)+1}$ {satisfy the conditions} 
		\begin{equation*}
			\begin{dcases}
				p_{k^\prime(n)} &\ge  p_{k^\prime(n)+1}, ~\text { and }\\
				p_{k^\prime(n)+1}&= \max_{k^\prime(n)+1 \le j \le n} p_j,
			\end{dcases}
		\end{equation*}
		{then $\tau^*$} is $\varepsilon$-optimal, where 
		$$0 \le \varepsilon \le   p_{k^\prime(n)}p_{k^\prime(n)+1}.$$
		Moreover, in the case where $\profile$ coincides with the Karamata-Stirling success profile, {one can choose}  $\displaystyle\varepsilon=O\left(\frac{1}{n^2}\right)$ and the threshold rule is asymptotically optimal.
	\end{thm}
	\begin{proof}
		Fix $m \in \{1, 2, \ldots, n-1\} $ and let $k_m$ be the mode's index of the probability function $s_m(\cdot, n)$. By the unimodality of this function, $s^*_m(n)= s_m(k_m, n)$  is the maximum probability of stopping on the $m$th  last success. Now, let $s_m(k^\prime(n), n)$ be the  probability of stopping at the threshold defined by the mean rule (\ref{def.threshold.darroch}). 
		We want to show that for every $n$, there exists an $\varepsilon \ge 0$ such that
		$$s^*_m(n) - \varepsilon \le s_m(k^\prime(n) , n) \le  s^*_m(n) ~~\text{ and }~~ \varepsilon \to 0 ~ \text{ as } n\to \infty. $$

		First, it is not difficult to show that $0 \le  k_m  - k^\prime(n)\le 1$. By hypothesis, it holds that $p_{k^\prime(n)}  >  p_{k^\prime(n)+1} - \displaystyle\min_{k^* \le j \le n} p_j$  and $p_{k^\prime(n)+1} = \displaystyle\max_{k^\prime(n)+1 \le j\le n} p_j$. By applying Theorem \ref{lem.samuels}, this implies that the mode can only be in two locations. In the first, $k^\prime(n) = k_m$  and hence $\varepsilon = 0$. Conversely,  we need to wait for one more trial, which has the index $k_m-1$, and then stop at the next success. Since, at most, we miss the optimal threshold by one trial, we may define $\varepsilon$ as follows 
		\begin{alignat*}{2}
			\varepsilon &= \Delta_k s_m(k_m, n)\\
			&= p_{k_m-1}\Delta_m s_m(k_m, n)\\
			&= p_{k_m-1}s_0(k_m, n)\left(  R_m(k_m, n) -R_{m-1}(k_m, n)\right)\\
			&\le p_{k_m-1}s_0(k_m, n)\left(  R_m(k_m, n) -R_{m-1}(k_m+1, n)\right) ~~(R_{m-1} \text{ decreasing in the 1st parameter}) \\
			&\le p_{k_m-1}s_0(k_m, n)\left(  R_m(k_m, n) -R_{m}(k_m+1, n)\right) ~~~~~ (R_m(k_m+1, n) \le R_{m-1}(k_m+1, n))\\						
			&= p_{k_m-1}s_0(k_m, n)\left(r_{k_m}R_{m-1}(k_m+1, n)\right) \qquad\qquad~~~ (\text{By } 	(\ref{Rj(k+1)}))\\												
			&= p_{k_m-1}s_{m-1}(k_m+1, n)p_{k_m}\\								
			&\le p_{k_m-1}p_{k_m}															
		\end{alignat*}
		Therefore, it holds that $$0 \le \varepsilon \le p_{k^\prime(n)}p_{k^\prime(n)+1}.$$
		
		For the second part of the theorem, we consider the Karamata-Stirling profile, where the success probabilities are given by $p_k = \theta / (\theta+k-1)$ for $k \ge 1$. These probabilities are decreasing and, therefore, satisfy the assumptions of the theorem. As $n \to \infty$, $k^\prime(n) \to \infty$ and both $p_{k^\prime(n)}=O(n^{-1})$ and $p_{k^\prime(n)+1}=O(n^{-1})$. Consequently, $\varepsilon = O(n^{-2}) \to 0$, and the threshold policy is asymptotically optimal, which is precisely the second assertion of the theorem.
	\end{proof}
	
	\section{Asymptotics}
	
	By exploiting the convergence of the Poisson-binomial distribution to the Poisson distribution, we derive an asymptotically optimal threshold strategy. Poisson approximation applied to the Poisson-binomial distribution case  has been covered extensively in the literature \cite{ BarbourHall, DeheuvelsPfeifer, LeCam, Wang}. For the specific case of the Karamata-Stirling profile, the approximation appeared in many sources. For instance, the work of Arratia et al.  \cite{Arratia} and a more recent publication by Yamato \cite{Yamato}, along with the references therein, provide detailed discussions on this topic.
	
	A sequence $(X_n)$ of random variables taking values in $\mz_+$ converges in distribution to $\mu$ if and only if 
	$\dtv (\mu_n, \mu) \to 0$, where $\mu_n$ is the law of $X_n$ \cite{Barbour}. Let $W_{n} = X_1 + \ldots + X_n$ be the number of successes in $n$ Bernoulli trials having success probabilities $p_1, \ldots, p_n$. For a positive real $\lambda$, let $Y_\lambda$ denote a Poisson-distributed random variable with mean parameter $0 < \lambda < \infty$. We represent this distribution by  $\Poisson(\lambda)$. We wish to approximate the distribution of $W_{n}$, denoted $\calAlph{L}(W_n)$, with that  of $Y_\lambda$. To this end, denote the mean of $W_n$ by $\lambda_{n} $ and the sum of the squares of  probabilities by $\lambda^{(2)}_{n}$. These two quantities are given explicitly for the Karamata-Stirling profile (\ref{profile}) by $	\lambda_{n, 1} $ and $	\lambda^{(2)}_{n, 1} $ respectively, where
	\begin{equation*}
		\begin{dcases}
			\lambda_{n, k} = \sum_{j=k}^n \frac{\theta}{\theta + j- 1} , \\
			\lambda^{(2)}_{n, k} = \sum_{j=k}^n \left(\frac{\theta}{\theta + j- 1} \right)^2.
		\end{dcases}
	\end{equation*}
	For the total variation distance between $\calAlph{L}(W_{n})$ and  $\Poisson(\lambda)$, the earliest result is due to Le Cam \cite{Barbour, LeCam}, who proved that 
	$$\dtv \left(\lawof{W_n}, \Poisson(Y) \right) \le \lambda^{(2)}_{n, k}.$$
	
	Where $\dtv$ is the the total variation distance defined for every measures $\mu_1, \mu_2$  in some set of probability measures on $\mz_+$ by
	\begin{alignat*}{2}
		\dtv(\mu_1, \mu_2) = \sup \{ \abs{\mu_1(A) - \mu_2(A)} : A \in \mz_+\} = \frac{1}{2} \sum_{j \ge 0}   \abs{\mu_1(j) - \mu_2(j)}.
	\end{alignat*}

	Barbour and Hall (1984) give lower  and upper bounds for the total variation distance as follows
	\begin{equation*}
		\frac{1}{32} \min\left( 1, \frac{1}{\lambda_{n,k}}\right) \lambda^{(2)}_{n, k} ~\le~ \dtv \left(\lawof{W_k}, \Poisson(Y) \right) ~\le~ \frac{1}{\lambda_{n,k}} \left( 1-e^{-\lambda_{n,k}}\right)\lambda^{(2)}_{n,k}.
	\end{equation*}
	
	A necessary and sufficient condition for the convergence in law of $W_n$ to  $\Poisson(Y)$ is clear from this inequality and is given by Wang (1993) as follows.
	\begin{assertion}[Theorem 2 in \cite{Wang}]\label{prop.wang}
		Consider a sequence of independent Bernoulli random variables with success probabilities $p_1, p_2, \ldots$ and let $W_n$ be the partial sum of the first $n$  terms. $W_n$ converges weakly to $\Poisson(\lambda)$, where $~0 < \lambda < \infty~$, if and only if 
		$\mE{W_n} \to \lambda$ and $\displaystyle\sum_{i \ge 1} p^2_i \to 0$ as $n \to \infty$.
	\end{assertion}
	The next result  is an adaptation of Proposition 3.1 in \cite{Yamato}, which provides conditions under which the number of successes converges to a Poisson distribution, symbolised $\Poisson(\cdot)$, but taking into account that a proportion of trials is to be skipped.
	
	\begin{lemma}\label{prop:poisson.approx}
		Consider the random variable $W_{k, n}$ that counts the number of successes  in  trials $X_k,  \ldots, X_n$  whose success probabilities are given by (\ref{profile}). $W_{n,k}$ converges in distribution to $\Poisson(\theta\log(1/\kappa))$  if and only if
		$$ \lambda_{n,k} \to \theta \log(1/\kappa) ~~~~\text{as} ~~~ n\to \infty,$$
		where $\displaystyle \kappa = \lim_{n \to \infty} \frac{k}{n}$ is  the proportion of trials to skip.
	\end{lemma}
	\begin{proof}
		Recall the inequality for $j \ge 1$
		$$\theta\int_{j}^{j+1} \frac{\dx}{x+ \theta} < \frac{\theta}{j + \theta} < \theta\int_{j-1}^{j}  \frac{\dx}{x+ \theta}.$$	
		Sum these inequalities from $k-1$ to $n-1$ to obtain
		\begin{equation*}
			\theta\log\left(\frac{n+\theta}{k+\theta}\right) < \lambda_{n,k} < \theta\log\left(\frac{n+\theta-1}{k+\theta-1}\right).
		\end{equation*}
		Taking the limit as $n \to \infty,~  k\to \infty$ and letting $\kappa = \lim_{n \to \infty} k/n$ yields
		$$\lambda_{n,k}  \to \theta\log(1/\kappa).$$

		Similarly,  for $j \ge 1$
		$$\theta^2\int_{j}^{j+1} \frac{\dx}{(x+ \theta)^2} < \frac{\theta^2}{(j+ \theta)^2} < \theta^2\int_{j-1}^{j}  \frac{\dx}{(x+ \theta)^2}.$$	
		Summing these inequalities from $k-1$ to $n-1$, we get
		\begin{equation*}
			\theta^2 \left( \frac{1}{n+\theta+1}-  \frac{1}{k+\theta}\right) < \lambda^{(2)}_{n,k} < \theta^2 \left( \frac{1}{n+\theta-1} - \frac{1}{k+\theta-1}\right).
		\end{equation*}
		As $n \to \infty,~ k\to \infty$, we immediately get $\lambda^{(2)}_{n,k}  \to 0$. Therefore, by Proposition \ref{prop.wang}, $W_{n,k}$ converges in law to $\Poisson(\theta\log(1/\kappa)).$
		This concludes the proof.	
	\end{proof}		
	\begin{rem}
		Without appealing to Wang's result, we may use the well-known fact that if $\sum_{j \ge k} p_j \to \infty $ and $\max \{ p_{k}, p_{k+1}, \ldots, p_n\} \to 0$ as $n \to 0$, which is satisfied in our case as $k$ depends on $n$, then weak convergence is achieved. This is a direct result of Theorem 1.2 of Deheuvels and Pfeifer (1986) \cite{DeheuvelsPfeifer, Pfeifer}.
	\end{rem}
	Having established weak convergence to a Poisson distribution, we solve the problem of stopping on the $m$th last success, asymptotically, as follows. 
	
	\begin{thm}[Asymptotic Rule]\label{thm.darroch.problem.poisson}
		For the problem of  selecting  the $m$th last success in a sequence of $n$ independent Bernoulli trials with success profile given by (\ref{profile}), the threshold policy, which stops at the first success, if any, on or after index $k_{\rm poi}(n)$ given by
		\begin{equation*}
			k_{\rm poi}(n) \coloneqq \min\{k \in [n]:    \frac{k}{n} \ge \exp\left(-\frac{m}{\theta}\right) \} \vee 1, 
		\end{equation*}
		is asymptotically optimal.	 Moreover, the asymptotic winning probability is $$ \frac{m^m}{m!} \exp(-m).$$
	\end{thm}
	\begin{proof}
		Based on Proposition \ref{prop:poisson.approx},  the sum of probabilities $W_{n,k} = X_k + \cdots + X_n$ converges to a Poisson random variable $Y$ with mean $\lambda = \theta\log\left(\frac{n}{k}\right)$. 
		When $n/k \to \infty$, the total variation distance between $W_{n,k}$ and $Y$ tends to 0. Therefore, the asymptotic winning probability, which is $\sup \proba(W_{n,k} = m)$,  tends to $$\sup \{~\frac{\lambda^m} { m! }\exp(-\lambda)~|~ \lambda > 0\}.$$ This supremum can be easily achieved when $\lambda = m$, and hence, the asymptotic winning probability is $ m^m \exp(-m) / m!.$
		To get the optimal threshold, set $\lambda = m$ and solve for the proportion of trials $\displaystyle\frac{k}{n}$ to skip, where $k=1, 2, \ldots, n$. 
		From equation $\theta\log\left(\displaystyle\frac{n}{k}\right) = m$, a straightforward calculation yields
		\begin{alignat*}{2}
			\frac{k}{n} = \exp\left(-\frac{m}{\theta}\right).
		\end{alignat*}
		This completes the proof.
	\end{proof}	
	\section{Numerical results}
	We use the Karamata-Stirling profile to calculate optimal probabilities for two scenarios: the last success ($ s*_1 (n)$) and the second-to-last success ($ s*_2 (n)$) problems. 
	
	The trial numbers used for the simulation are listed as $n \in \{10, 100, 1000, 10000, 100000\}$. In the tables below, $m\in\{1,2\}$ indicates the number of successes, while  $\theta \in \{0.5, 1, 1.5, 2\}$ represents the mutation parameter of the Karamata-Stirling success profile.
	
	As benchmark values, recall that the asymptotic winning probability for the last success problem is $1/e$, and for the second-to-last success problem, it is $2e^{-2} \sim 0.270670$. The classical best choice problem aligns with the setting ($\theta=1$ and $m=1$), highlighted in grey in the first two tables. Finally, $\varepsilon$ is the discrepancy between the two stopping rules (Optimal - Suboptimal).

	\begin{table}[H]
		\begin{tabular}{c c >{\columncolor{Gray}}c c c c  >{\columncolor{Gray}}c c c}
			\toprule
			\multirow{3}{*}[-2pt]{n} &\multicolumn{4}{c}{$\varepsilon$} & \multicolumn{4}{c}{$s^*_1(n)$} 
			\\\cmidrule(lr){2-5}\cmidrule(lr){6-9}
			&\multicolumn{4}{c}{$\theta$}&\multicolumn{4}{c}{$\theta$}
			\\\cmidrule(lr){2-5}\cmidrule(lr){6-9}	
			&0.5 & 1 & 1.5 & 2 & 	0.5 & 1 & 1.5 & 2 	 \\
			\hline
			\midrule
			10 & 0.000000 & 0.000000 & 0.011817 & 0.009957 & 0.401393 & 0.398690 & 0.409131 & 0.416667 \\
			100 & 0.000000 & 0.000028 & 0.000184 & 0.000000 & 0.370812 & 0.371043 & 0.371823 & 0.372649 \\
			1000 & 0.000000 & 0.000001 & 0.000002 & 0.000003 & 0.368174 & 0.368196 & 0.368272 & 0.368357 \\
			10000 & 0.000000 & 0.000000 & 0.000000 & 0.000000 & 0.367909 & 0.367911 & 0.367919 & 0.367927 \\
			100000 & 0.000000 & 0.000000 & 0.000000 & 0.000000 & 0.367882 & 0.367883 & 0.367883 & 0.367884 \\
			\bottomrule
		\end{tabular}
		\label{m1_table}
		\caption{Performance of the mean rule versus the odds rule for $m=2$.}
	\end{table}
	
	In the next table,  thresholds generated by the odds rule and by the mean rule are compared. Discrepancies between the thresholds, optimal versus suboptimal, are highlighted in red.
	
	\begin{table}[H]
		\centering
		\begin{tabular}{c l >{\columncolor{Gray}}l l l l >{\columncolor{Gray}}l l l}
			\toprule
			\multirow{3}{*}[-2pt]{n} &\multicolumn{4}{c}{Odds} & \multicolumn{4}{c}{Probabilities} 
			\\\cmidrule(lr){2-5}\cmidrule(lr){6-9}
			&\multicolumn{4}{c}{$\theta$}&\multicolumn{4}{c}{$\theta$}
			\\\cmidrule(lr){2-5}\cmidrule(lr){6-9}	
			&0.5 & 1 & 1.5 & 2 & 	0.5 & 1 & 1.5 & 2 	 \\
			\hline
			\midrule
			10 & 2 & 4 & 6 & 7 & 2 & 4 & 5 &  \color{red}6 \\
			100 & 14 & 38 & 52 & 61 & 14 &  \color{red}37 &  \color{red}51 & 61 \\
			1000 & 136 & 369 & 514 & 607 & 136 &  \color{red}368 &  \color{red}513 & \color{red}606 \\
			10000 & 1354 & 3680 & 5135 & 6066 & 1354 &  \color{red}3679 &  \color{red}5134 &  \color{red}6065 \\
			100000 & 13534 & 36789 & 51342 & 60654 & 13534 &  \color{red}36788 & 51342 &  \color{red}60653 \\
			\bottomrule
		\end{tabular}
		\label{m1_table_thresh}
		\caption{Thresholds of the mean and odds rules for $m=1$.}
	\end{table}
	
	The next two tables detail the results for the second-to-last success problem ($m=2$).
	\begin{table}[H]
		\begin{tabular}{c c c c c c c c c}
			\toprule
			\multirow{3}{*}[-2pt]{n} &\multicolumn{4}{c}{$\varepsilon$} & \multicolumn{4}{c}{$s^*_2(n)$} 
			\\\cmidrule(lr){2-5}\cmidrule(lr){6-9}
			&\multicolumn{4}{c}{$\theta$}&\multicolumn{4}{c}{$\theta$}
			\\\cmidrule(lr){2-5}\cmidrule(lr){6-9}	
			&0.5 & 1 & 1.5 & 2 & 	0.5 & 1 & 1.5 & 2 	 \\
			\hline
			\midrule
			10 & 0.000000 & 0.323165 & 0.025088 & 0.016082 & 0.000000 & 0.323165 & 0.319856 & 0.322121 \\
			100 & 0.000588 & 0.000000 & 0.000077 & 0.000191 & 0.279770 & 0.275097 & 0.274961 & 0.275371 \\
			1000 & 0.000000 & 0.000002 & 0.000002 & 0.000002 & 0.271589 & 0.271104 & 0.271096 & 0.271136 \\
			10000 & 0.000000 & 0.000000 & 0.000000 & 0.000000 & 0.270761 & 0.270714 & 0.270713 & 0.270717 \\
			100000 & 0.000000 & 0.000000 & 0.000000 & 0.000000 & 0.270680 & 0.270675 & 0.270675 & 0.270675 \\
			\bottomrule
		\end{tabular}
		\label{m2_table}
		\caption{Performance of the mean rule versus the odds rule for $m=2$.}
	\end{table}
	\begin{table}[H]
		\centering
		\begin{tabular}{c l l l l l l l l}
			\toprule
			\multirow{3}{*}[-2pt]{n} &\multicolumn{4}{c}{Odds} & \multicolumn{4}{c}{Probabilities} 
			\\\cmidrule(lr){2-5}\cmidrule(lr){6-9}
			&\multicolumn{4}{c}{$\theta$}&\multicolumn{4}{c}{$\theta$}
			\\\cmidrule(lr){2-5}\cmidrule(lr){6-9}	
			&0.5 & 1 & 1.5 & 2 & 	0.5 & 1 & 1.5 & 2 	 \\
			\hline
			\midrule
			10 & 1 & 2 & 3 & 4 & 1 &  \color{red}1 &  \color{red}2 &  \color{red}3 \\
			100 & 3 & 14 & 27 & 37 &  \color{red}2 & 14 &  \color{red}26 &  \color{red}36 \\
			1000 & 19 & 136 & 264 & 368 & 19 &  \color{red}135 &  \color{red}263 &  \color{red}367 \\
			10000 & 184 & 1354 & 2636 & 3679 & 184 &  \color{red}1353 & 2636 &  \color{red}3678 \\
			100000 & 1832 & 13534 & 26360 & 36788 & 1832 & 13534 &  \color{red}26359 &  \color{red}36787 \\
			\bottomrule
		\end{tabular}
		\label{m2_table_thresh}
		\caption{Thresholds of the mean and odds rules for $m=2$.}
	\end{table}
	
	\section*{Acknowledgment}	
I would like to thank Sasha Gnedin for his insightful suggestions and for introducing me to Samuels' work. I am also grateful to Ross Pinsky for identifying errors in an earlier draft of this article.
	\printbibliography
	
\end{document}